\documentclass[10pt,electronic]{amsart}        
\usepackage[applemac]{inputenc}
\usepackage[T1]{fontenc}
\usepackage[small,it]{caption}
\usepackage[english]{babel}
\usepackage{url}
\usepackage[hmargin=3.7cm,vmargin=3.5cm]{geometry} 
\usepackage{graphicx}                         
\usepackage{amsmath}
\usepackage{amsthm}                      
\usepackage{amssymb}
\usepackage{amscd}     
\usepackage{mathrsfs}
\usepackage{stmaryrd}
\usepackage{enumerate}

\usepackage{euscript}
\renewcommand{\mathcal}{\EuScript}

\theoremstyle{plain}                                                           
\newtheorem{thm}{Theorem}[section]

\newtheorem{cor}[thm]{Corollary}

\theoremstyle{definition}

\newtheorem{rem}[thm]{Remark}

\newcommand{\F}{\mathcal F}
\newcommand{\field}[1]{\ensuremath{\mathbf{#1}}}
\newcommand{\Q}{\ensuremath{\field{Q}}}        
\newcommand{\Z}{\ensuremath{\field{Z}}} 
\newcommand{\M}{\mathcal{M}}
\newcommand{\Pic}{\operatorname{Pic}}

\newcommand{\IH}{{I\! H}}
\newcommand{\RH}{{R\hspace{-0.5pt} H}}
\newcommand{\CH}{\mathrm{C\hspace{-0.5pt} H}}

\title{A vanishing result for tautological classes on the moduli of K3 surfaces}

\author{Dan Petersen}
\thanks{}
\email{danpete@math.ku.dk}
\address{Institut for Matematiske Fag \\ Universitetsparken 5 \\ 2100 K{\o}benhavn \O \\ Denmark}

\begin{document} 
 \maketitle   
 
 \begin{abstract} Looijenga's vanishing theorem on the moduli space of curves $\M_g$ says that the tautological ring vanishes above degree $g-2$. We prove an analogous result for the tautological cohomology ring of the moduli space of K3 surfaces. \end{abstract}

\section{Introduction}

Let $\M_g$ be the moduli space of smooth curves of genus $g$. The \emph{tautological ring} $R^\bullet(\M_g)$ is a subring of its Chow ring $\CH^\bullet_\Q(\M_g)$ generated by particularly natural algebraic cycle classes. This tautological ring has been intensely studied in the past decades, in particular because of the influential series of conjectures of Faber \cite{faberconjectures}. The first major result about the tautological ring is Looijenga's theorem \cite{looijengatautological} (whose statement was one Faber's conjectures), which says that $R^k(\M_g) = 0$ for $k>g-2$ and that $R^{g-2}(\M_g)\cong \Q$; thus the tautological ring is significantly smaller than one would expect a priori, since the Chow ring itself will certainly not have any such vanishing properties in general.

Let $\F_g$ be the moduli space of quasi-polarized K3 surfaces of genus $g$. The intersection theory of the moduli spaces $\F_g$ has been the subject of recent attention \cite{gwnl,bergeronlimillsonmoeglin,marianopreapandharipande,pandharipandeyin}, and tautological rings $R^\bullet(\F_g)$ have been defined by analogy with the curve case. The first proposed definition was due to van der Geer and Katsura \cite{vandergeerkatsura}, who suggested that the subring generated by $\lambda$ (the first Chern class of the Hodge line bundle) should be the natural tautological ring in this case, and they proved that this ring equals $\Q[\lambda]/(\lambda^{18})$. But it seems more natural to include also classes of \emph{special cycles} inside $\F_g$ as part of the tautological ring.

To be more precise, let $L$ be the lattice $H^2(X,\Z)$, where $X$ is a K3 surface. Let $\Lambda \subset L$ be a sublattice of signature $(1,19-d)$. We say that a \emph{$\Lambda$-polarized K3 surface} is a K3 surface $X$ equipped with a primitive embedding $\Lambda \hookrightarrow \Pic(X)$ for which the image of $\Lambda$ contains a class $D$ with the property that $D^2>0$ and $D \cdot C \geq 0$ for all curves $C \subset X$ (that is, $D$ is a quasi-polarization of $X$). Let $M_\Lambda$ denote the moduli space of $\Lambda$-polarized K3 surfaces. It is an orthogonal modular variety of dimension $d$. For $d=19$ we recover the moduli spaces $\F_g$. If $\Lambda \subset \Lambda'$ then $M_{\Lambda'} \subset M_\Lambda$. The subvarieties $M_{\Lambda'} \subset M_\Lambda$ are called \emph{Noether--Lefschetz loci} in $M_\Lambda$. Those of codimension one are called Noether--Lefschetz divisors. 

We define the tautological ring of $M_\Lambda$ to be the subring of $\CH^\bullet_\Q(M_\Lambda)$ generated by the classes of all Noether--Lefschetz loci. The resulting tautological ring would a priori seem to be smaller than the one defined in \cite[Section 4]{marianopreapandharipande} (which includes certain kappa classes), but Pandharipande and Yin \cite{pandharipandeyin} have recently proved that the two notions in fact agree, as conjectured in \cite[Conjecture 2]{marianopreapandharipande}. The same result has been independently obtained by Bergeron and Li (pers.\ comm.) for the image of the tautological ring in cohomology. 

The goal of this note is to prove the analogue of the vanishing part of Looijenga's theorem for the tautological ring of $\F_g$ in cohomology: more precisely, if $\RH^\bullet(\F_g)$ denotes the image of the tautological ring in cohomology, then we prove that $\RH^k(\F_g)=0$ for $k>34$.

\textbf{Acknowledgements.} I am grateful to Zhiyuan Li for nice discussions on these topics at Oberwolfach. 

\section{The main result}We keep the notation from the introduction. Let us recall the proof of the result of van der Geer and Katsura mentioned previously.
\begin{thm}[van der Geer--Katsura]If $\dim M_\Lambda = d$, $3 \leq d \leq 19$, then $\lambda^{d-1} =0$ and $\lambda^{d-2} \neq 0$ in $H^\bullet(M_\Lambda,\Q)$. 
\end{thm}

\begin{proof}
	We begin with the vanishing result. One proceeds by induction on $d$. When $d=3$, $M_\Lambda$ is a Siegel threefold and it is well known that $\lambda^2=0$ in this case. 
	
	For $d>3$, the class $\lambda^{d-2}$ multiplied by any Noether--Lefschetz divisor on $M_\Lambda$ vanishes, by the induction hypothesis. But $\lambda$ itself is a linear combination of classes of Noether--Lefschetz divisors \cite[Theorem 1.2]{bkps} --- more generally, every class in $H^2(M_\Lambda,\Q)$ is in fact a linear combination of Noether--Lefschetz divisors, after \cite{bergeronlimillsonmoeglin}. Thus $\lambda^{d-1}=0$. 
	
	For the nonvanishing result, one uses that the Baily--Borel compactification $M_\Lambda \subset M_\Lambda^\ast$ is a projective variety, and that $M_\Lambda^\ast \setminus M_\Lambda$ is of dimension one. It follows that the intersection of $M_\Lambda^\ast$ with two general hyperplanes gives a compact subvariety $Y \subset M_\Lambda$ of dimension $d-2$. Since $\lambda$ is ample on $M_\Lambda$, it follows that $\lambda^{d-2}$ is nonzero on $Y$, so $\lambda^{d-2}$ can not vanish.\end{proof}

\begin{thm}If $\dim M_\Lambda = d$, $4 \leq d \leq 19$, then $W_k H^k(M_\Lambda,\Q) =0$
	for $k > 2(d-2)$, where $W_\bullet$ denotes the weight filtration on the mixed Hodge structure on the cohomology of $M_\Lambda$. 
\end{thm}

\begin{proof}
	Pick $\alpha$ in $W_k H^k(M_\Lambda,\Q)$. The fact that $\alpha$ is of pure weight implies that it is in the image of the map $$ \IH^k(M_\Lambda^\ast,\Q) \to H^k(M_\Lambda,\Q)$$
	from intersection cohomology, see e.g.\ \cite[Lemma 2]{durfee}. Since intersection cohomology satisfies Hard Lefschetz and $\lambda$ is ample on $M_\Lambda^\ast$, we can write $\alpha = \lambda^{k-d}\cdot \beta$ with $\beta \in H^{2d-k}(M_\Lambda,\Q)$. Since $k>2(d-2)$, $\beta$ has degree at most $3$. The only nontrivial case is when $\beta \in H^2(M_\Lambda,\Q)$. In this case, $\beta$ is a linear combination of Noether--Lefschetz divisors  \cite{bergeronlimillsonmoeglin}, and $\alpha = \lambda^{d-2}\cdot \beta$ must vanish, as observed in the proof of the previous theorem.
\end{proof}

\begin{cor}
	The tautological cohomology ring $\RH^\bullet(\F_g)$ of the moduli space of polarized K3 surfaces vanishes above cohomology degree $34$. 
\end{cor}

\begin{proof}
	A tautological class is in particular the class of an algebraic cycle, so it lies in the pure part of the mixed Hodge structure. Thus the vanishing follows from the case $d=19$ of the previous theorem.
\end{proof}

\begin{rem}
	According to \cite{bergeronlimillsonmoeglin}, $H^k(\F_g,\Q)$ is tautological for $k\leq 8$. Therefore the same Hard Lefschetz argument shows more generally that every class in $W_k H^k(\F_g,\Q)$ is tautological for $k \geq 30$. 
\end{rem}

\begin{rem}
	It is natural to ask whether the tautological ring of $\F_g$ is one-dimensional in the top degree. There seems to be no compelling reason to believe this, however, other than by analogy with Looijenga's theorem about the tautological ring of the moduli space of curves $\mathcal M_g$. By the previous remark, this question (in cohomology) is equivalent to asking whether $W_{34}H^{34}(\F_g,\Q) \cong \Q(-17)$, or if it is larger. This seems like a more approachable problem, e.g.\ via the general theory of boundary cohomology/Eisenstein cohomology of Shimura varieties. 
\end{rem}

%\printbibliography

\bibliographystyle{alpha}
\bibliography{../database}

\end{document}